\documentclass{article}
\usepackage{graphicx,color,mathptmx,latexsym,amsmath, amsthm, amscd, amsfonts, amssymb}
 \numberwithin{equation}{section}
\usepackage{tikz}
\usepackage{pgf,tikz}
\usetikzlibrary{arrows}
\usepackage{calc}
\usepackage{tikz}
\usetikzlibrary{decorations.markings}
\tikzstyle{vertex}=[circle, draw, inner sep=0pt, minimum size=6pt]
\newcommand{\vertex}{\node[vertex]}
\newtheorem{theorem}{Theorem}[section]

\newtheorem{lemma}[theorem]{Lemma}
\newtheorem{corollary}[theorem]{Corollary}

\newtheorem{thm}{Theorem}[section]

\begin{document}
%\noindent  {\bf \footnotesize \emph{$9^{th}$ Iranian Group Theory Conference},
%{ University of Kashan, 13-15 Bahman, 1395 (February 1-3 2017), pp:00-00} }
%%%%%%%%%%%%%%%%%%%%%%%%%%%%%%%%%%%%%%%%%%%%%%%%%%%%%%%%%%%%%%%
\thispagestyle{plain}
%\vspace*{1.5cm} \\
\begin{center}
% Put the title of the paper below
{\huge A characterization of  $2$-dimensional Cayley graphs on dihedral groups}
\bigskip
\vspace*{1cm} \\

% Put the speaker's name below in bold.
A. Behtoei\footnote{Corresponding author}\\
 %Address of the first author :
 Department of Mathematics, Imam Khomeini International University, \\
 P. O. Box: 34149-16818, Qazvin, Iran, a.behtoei@sci.ikiu.ac.ir
%\vspace*{0.5cm} \\
 %S. B. Pejman \\
 %Address of the second author :
 %Faculty of Sciences, Imam Khomeini International University,  \\
 %b.pejman@edu.ikiu.ac.ir and
 \vspace*{0.5cm} \\
 Y. Golkhandy Pour\\
 %Address of the second author :
Department of Mathematics, Imam Khomeini International University,   \\
P. O. Box: 34149-16818, Qazvin, Iran,  y.golkhandypour@edu.ikiu.ac.ir

\end{center}
%%%%%%%%%%%%%%%
\vspace*{1.5cm}
%%%%%%%%%%%%%%%%%%
% Include the text of the abstract here.
% It should be no more than one page.
%%%%%%%%%%%%%%%%%%%%%%%%%%%%%%%%%%%%%%%%%%%%%%%%%%%%%%%%%%%%%%%
%put the abstract bellow.
\begin{abstract}
%A subset $W$ of $V(\mathcal{G})$ is called a resolving set where for each $u, v\in V(\mathcal{G})$ there exist $w\in W$ such that $\partial(u,w)\neq \partial(v, w)$. The metric dimension Of $\mathcal{G}$ is the minimum cardinality of a resolving set and is denoted by $dim(\mathcal{G})$.
A subset $W$ of the vertices of $G$ is  a  resolving set for $G$ when
for each pair of distinct vertices $u,v\in V(G)$ there exists $w\in W$
such that $d(u,w)\neq d(v,w)$,  where $d(x,y)$ denotes the distance
between  two vertices $x$ and $y$.
The cardinality of a minimum resolving set for $G$ is  the  metric dimension of $G$. %and is denoted by $\dim_M(G)$.
This concept arise in many diverse areas
including network discovery and verification, robot navigation, problems of pattern recognition and image processing, coin weighing problems,
strategies for the Mastermind game,
combinatorial search and optimization. The problem of finding metric dimension is NP-complete for general graphs.
In this paper, we characterize all of Cayley graphs on dihedral groups with metric dimension is two.
%$D_{2n}=\langle a,b|~a^n=b^2=1,~a^b=a^{-1} \rangle$ whose he metric dimension . More precisely,  we prove that $dim(Cay(D_{2n},S))=2$ if and only if $n$ is odd and $S=\{a^i, a^{-i}, b\}$ or $S=\{a^i, a^{-i}, a^jb\}$. ???????????????
\bigskip \\
\textbf{Keywords:}  Metric dimension, Cayley graph, Dihedral group.  \bigskip \\
\textbf{MSC(2010):}  Primary: 05C25, 05C75, 20B10.

\end{abstract}
%%%%%%%%%%%%%%%%%%%%%%%%%%%%%%%
%put the introduction bellow.
\section{Introduction}

Let $\Gamma=(V,E)$ be a simple and connected graph with vertex set $V$ and edge set $E$. The distance between two vertices $x,y\in V$ is the length of a shortest path between them and is denoted by $d(x,y)$. If $d(x,y)=1$, then for convenient we write $x\sim y$.
%,  otherwise we write $u\nsim v$.
The neighborhood of $x$ is $N(x)=\{y:~x\sim y\}$. % and the diameter of $\Gamma$ is  $diam(\Gamma)=\max\{ d(x,y):~x,y\in V\}$.
%For an ordered subset of vertices $W=\{w_1, w_2, ..., w_k\}$, the {\it metric representation} of a vertex $x$ with respect to $W$ is the $k$-vector
%$ r(x\mid W) = (d(x, w_1), d(x, w_2), ..., d(x, w_k))$.
%If distinct vertices of $\Gamma$ have distinct metric representations, then $W$ is called a {\it resolving set}.
%The {\it metric dimension} of $\Gamma$ is the minimum cardinality among all of resolving sets for $\Gamma$
%and is denoted by $dim(\Gamma)$.
%%%%%%%
%Throughout this paper all graphs are finite, simple and undirected. The notions $\delta$, $\Delta$ and $N_G(v)$ stand for minimum degree, maximum degree and the set of neighbours of vertex $v$ in $G$, respectively.
For an ordered subset $W=\{w_1,w_2,\ldots,w_k\}$ of vertices and a
vertex $v\in V$, the $k$-vector
$r(v|W):=(d(v,w_1),d(v,w_2),\ldots,d(v,w_k))$ is  called  the
\textit{metric representation} of $v$ with respect to $W$. The set $W$ is
called  a \textit{resolving set} for $\Gamma$ if distinct vertices of $\Gamma$ have
distinct representations with respect to $W$.
%We say a set $S\subseteq V$ \textit{resolves} a set $T\subseteq V$ if for each pair of distinct vertices $u$ and $v$ in $T$ there is a vertex $s\in S$ such that $d(u,s)\neq d(v,s)$.
Each  minimum resolving set is a \textit{basis}  and the \textit{metric dimension} of $\Gamma$, $\dim_M(\Gamma)$, is
the cardinality of a basis  for $\Gamma$.
% A graph with metric dimension $k$ is called $k$-\textit{dimensional}.
This concept has applications in many areas
including network discovery and verification~\cite{net2}, robot navigation~\cite{landmarks}, problems of pattern recognition and image processing~\cite{digital},
coin weighing problems~\cite{coin},
strategies for the Mastermind game~\cite{Mastermind},
combinatorial search and optimization~\cite{coin}.
These concepts were introduced by  Slater in 1975 when he was working with U.S. Sonar and Coast Guard Loran stations and he described the usefulness of these concepts, \cite{Slater1975}.
Independently, Harary and Melter \cite{Harary} discovered these
concepts. Finding families of graphs with constant metric dimension or characterizing $n$-vertex graphs with a specified metric dimension are fascinating problems and atracts the attention of many researchers.
The problem of finding metric dimension is NP-Complete for general graphs  but the metric dimension of trees can be  obtained  using a polynomial time algorithm.
In \cite{European-j-comb1} the metric dimension for two different models of random forests is investigated.
 It is not hard to see that for each $n$-vertex graph $\Gamma$ we have $1\leq \dim_M(\Gamma) \leq  n-1$.
  Chartrand et al.~\cite{Ollerman} proved that for $n\geq 2$, $\dim_M(\Gamma)=n-1$
if and only if $\Gamma$ is the complete graph $K_n$. The metric dimension of each complete $t$-partite graph with $n$ vertices is $n-t$. They also provided a
 characterization of  graphs of order $n$ with
metric dimension $n-2$, see \cite{Ollerman}.  Graphs of order $n$  with metric dimension $n-3$ are characterized in~\cite{n-3}.
  Khuller et al.~\cite{landmarks} and Chartrand et al.~\cite{Ollerman}
  proved that $\dim_M(\Gamma)=1$ if and only if $\Gamma$ is a path $P_n$.
  The metric dimension of the power graph of a cyclic group is determined in \cite{European-j-comb2}. 
  Salman $et~al.$ studied this parameter for the Cayley graphs on cyclic groups \cite{2}.
  Each cycle graph $C_n$ is a  $2$-dimensional graph ($\dim_M(C_n)=2$). All of  2-trees with metric dimension two are characterized in \cite{2-tree} and $2$-dimensional Cayley graphs on Abelian groups are characterized in \cite{CayAbel}.  
Moreover,  in~\cite{chang}  some properties of $2$-dimensional graphs are obtained as below.

\begin{thm}\label{thm:degree of basis elements}{\em\cite{chang}}
Let $\Gamma$ be a $2$-dimensional graph. If  $\{u,v\}$ is a basis for $\Gamma$, then
\begin{enumerate}
\item
there is a unique shortest path $P$ between $u$ and $v$,
\item
the degrees of $u$ and $v$ are at most three,
\item
 the degree of each internal vertex  on $P$ is  at most five.
\end{enumerate}
\end{thm}
%%%%%%%
The $M\ddot{o}bius~ Ladder$  graph $M_n$ is a cubic circulant graph with an even number $n$ of vertices  formed from an $n$-cycle by connecting opposite pairs of vertices in the cycle. For the metric dimension of $M\ddot{o}bius~ Ladders$ we have the following result.

\begin{theorem} \cite{7}  \label{21}
Let $n\geq 8$ be an even number. The metric dimension of each Mobius Ladder $M_n$ is 3 or 4. Specially,  $dim_M(M_n)=3$ when $n\equiv 2 ~(\!\!\!\!\!\mod 8)$.
\end{theorem}
C$\acute{a}$ceres $et~al.$ studied  the metric dimension of the Cartesian product of graphs.
\begin{theorem}\cite{cartesian product}\label{11}
 Let $P_m$ be a  path on $m\geq 2$ vertices and $C_n$ be a cycle on $n\geq 3$ vertices. Then the metric dimension of each prism $P_m\times C_n$ is given by
\begin{eqnarray*}
dim_M(P_m\times C_n)=
      \begin{cases}
       2 & ~~ n~~ odd,  \\
       3 & ~~ n~~ even.
      \end{cases}
\end{eqnarray*}
%Also, for $m\geq2$ and $n\geq2$  the metric dimension of each grid $P_m\times P_n$ is two.
\end{theorem}

%For more results in this subject or related subjects see \cite{6}, \cite{9},  \cite{5} and \cite{4}.

Let $G$ be a group and let $S$ be a subset of $G$ that is closed under taking inverse and does not contain the identity element, say $e$.
Recall that the {\it Cayley graph} $Cay(G,S)$ is a graph whose vertex set is $G$ and two vertices $u$ and $v$ are adjacent in it when $uv^{-1}\in S$.
Since $S$ is inverse-closed ($S=S^{-1}$) and does not contain the identity, $Cay(G,S)$ is a simple graph.
It is well known that $Cay(G,S)$ is a connected graph if and only if $S$ is a generating set for $G$.
%Let $H$ be a proper subgroup of $G$ and $T=\{e=g_1, g_2, \ldots, g_t\}$ be a right transversal for $H$ in $G$, where $t=[G:H]$.
%It is not hard to see that for $S=G\setminus H$ we have $G=\langle S\rangle$ and hence, $Cay(G,S)$ is a connected graph.
%If $u,v\in Hg_i$ for some $1\leq i\leq t$, then $uv^{-1}\in H$ which implies that $u$ and $v$ are not adjacent in $Cay(G,S)$.
%Thus, each right coset of $H$ induces an independent set of vertices. Also, since $Cay(G,S)$ is a connected and $|S|$-regular graph, the Cayley graph $Cay(G,S)$ is isomorphic to a complete $t$-partite graph whose partite sets are of size $|H|$.
%Therefore, we obtain the following results about a family of Cayley graphs whose metric dimension depends only on the size of its subgroups.
%\begin{proposition}\label{19}
%Let $H$ be a proper subgroup of $G$ and $S=G\setminus H$. Then we have  $dim(Cay(G,S))=|G|-[G:H]$.
%\end{proposition}
%
%
%\begin{theorem}  \cite{2}  \label{12}
%Let $G=\langle g\rangle$ be a cyclic group of even order $n$ and $S\subset G$ such that $e\notin S=S^{-1}$. If $S=\{g, g^{-1}, g^{n/2}\}$,
%then ????????????????
%\begin{equation*}\label{13}
%dim(Cay(G,S))=
 %  \begin{cases}
  %   3  & ~~ n\equiv 0~ (\!\!\!\!\!\!\!\!\!\mod 4), \\
   %  4  & ~~ n\equiv 2~ (\!\!\!\!\!\!\!\!\!\mod 4).
 %  \end{cases}
%\end{equation*}
%\end{theorem}
In this paper, we study the metric dimension of Cayley graphs on dihedral groups and we characterize all of Cayley graphs on dihedral groups whose metric dimension is two.

%%%%%%%%%%%%%%%%%%%
%%%%%%%%%%%%%%%%%%%
%%%%%%%%%%%%%%%%%%%

\section{\bf{Main results}}
%Ba\v{c}a $et~al.$ in 2011 determined the metric dimension of some $n$-vertex $k$-regular bipartite graphs for $k=n-1$ and $k=n-2$, see \cite{103}.
At first, we provide two useful  lemmas on dihedral groups  and a  sharp lower bound for the metric dimension of $3$-regular bipartite graphs which will be used in the sequel.% which will frequently be used in sequel.

%%%%%%%%%%%
%\begin{lemma}\label{26}
%  Let $G$ be a group and $S$ be an inverse-closed generating subset of $G$. If $s\in S$ and $O(s)\geq3$, then $Cay(G,S)$ contains a cycle with $O(s)$ vertices.
%\end{lemma}
%\begin{proof}
%It is easy to see that for each $s\in S$ with $O(s)=m$, vertices  $e, s, s^2, ..., s^{m-1}$ induce a cycle on $m$ vertices in $Cay(G,S)$.
%\end{proof}

\begin{lemma} \label{GenCond}
The subset $\{a^ib,a^jb\}$ is a generating set for dihedral group  $D_{2n}=\langle a, b|~a^n=b^2=(ab)^2=e\rangle$ if and only if $\gcd(n,i-j)=1$.
\end{lemma}
\begin{proof}
It is strightforward using the following fact about the subgroup generated by these elements.
$$\langle a^ib,a^jb\rangle =\{a^{(i-j)t},~a^{(i-j)t+i}b,~a^{(i-j)t+j}b~|~t\in \Bbb{Z} \}.$$
\end{proof}

\begin{lemma} \label{4|n}
If $4\mid n$ and $\gcd(i-j,n)=2$, then $\{ a^{n\over 2},a^ib,a^jb\}$ is not a generating set for $D_{2n}$.
\end{lemma}
\begin{proof}
Since $\langle a^{i-j}\rangle$ and $\langle a^2\rangle$ are two cyclic subgroups of order ${n\over 2}$ in the cyclic group $\langle a\rangle$, we have $\langle a^2\rangle=\langle a^{i-j}\rangle\subseteq \langle \{ a^{n\over 2},a^ib,a^jb\} \rangle$.
Since $4\mid n$, $a^{n\over 2}\in \langle a^2\rangle$ and hence, $\langle \{ a^{n\over 2},a^ib,a^jb\} \rangle=\langle \{ a^ib,a^jb\} \rangle$.
Now the result follows from Lemma \ref{GenCond}.
\end{proof}

\begin{lemma}\label{22}
Let $\Gamma$ be a $3$-regular bipartite graph on $n$ vertices. Then $dim_M(\Gamma)\geq 3$.
\end{lemma}
\begin{proof}
Since $\Gamma$ is not a path, $dim_M(\Gamma)$ is at least two.
Suppose that $dim_M(\Gamma)=2$ and let $W=\{u, v\}$ be a resolving set for $\Gamma$.
Assume that $d(u, v)=d$ and $N(u)=\{u_1, u_2, u_3\}$. It is easy to see that $d(u_i, v)\in \{d-1, d, d+1\}$, for each
 $1\leq i\leq 3$. If there exist $1\leq i<j\leq 3$ such that $d(u_i, v)=d(u_j, v)$, then $r(u_i|W)=r(u_j|W)$, which is a contradiction.
 Hence, without loss of generality, we can assume that
 $$d(u_1, v)=d-1,~d(u_2, v)=d, ~d(u_3, v)=d+1.$$
Let  $\sigma_1$ be a (shortest) path between two vertices $u$ and $v$ of lengh $d$, and  $\sigma_2$ be a (shortest) path between two vertices $u_2$ and $v$.
Two paths $\sigma_1$ and $\sigma_2$ using the edge $uu_2$ produce
an old closed walk of lengh $2d+1$ in  $\Gamma$  which contradicts the fact that $\Gamma$ is a bipartite graph. %, see \text{[\cite{8}, page~24]}.
 For the sharpness of this bound, consider the hypercube $Q_3=K_2\times K_2\times K_2$.
\end{proof}

In Theorem \ref{dih} we characterize  all of Cayley graphs on dihedral groups  whose metric dimension is two.  Recall that the center of $D_{2n}$ is  $\langle a^{n\over2}\rangle$ when $n$ is even,  otherwise it is the trivial subgroup $\{e\}$.

\begin{theorem}\label{dih} Let $S$ be a generating subset of $D_{2n}=\langle a, b|~a^n=b^2=(ab)^2=e\rangle$ such that $e\notin S=S^{-1}$.
Then we have $dim_M(Cay(D_{2n},S))=2$ if and only if  one of the following situations occures.
\begin{itemize}
\item[a)] $n=|S|=2$,
\item[b)] $S=\{a^ib,a^jb\}$ and $\gcd(i-j,n)=1$,
\item[c)] $n$ is odd and $S=\{a^i,a^{-i},a^jb\}$ with $gcd(i,n)=1$ and $j\in\{1,2,...,n\}$.
\end{itemize}
\end{theorem}

\begin{proof}
Since $D_{2n}$ is not a cyclic group, we have $|S|\geq 2$. Assume that $|S|=2$ and $S=\{x,y\}$. Since $S=S^{-1}$ and $D_{2n}$ is not cyclic, we have  $y\neq x^{-1}$ and $x^2=y^2=e$. If $S=\{a^{n\over 2},a^jb\}$ for some $1\leq j\leq n$, then the condition $D_{2n}=\langle S\rangle$ implies that $n=2$ and $D_{2n}=D_4$.
Otherwise, $S=\{a^ib,a^jb\}$ and using Lemma  \ref{GenCond} we have $\gcd(i-j,n)=1$. Hence, $Cay(D_{2n},S)$ is a connected $2$-regular graph (a cycle) and  $dim_M(Cay(D_{2n},S))=2$. If $|S|\geq 4$, then the degree of each vertex in $Cay(D_{2n},S)$ is at least $4$ and part (2) of Theorem \ref{thm:degree of basis elements} implies that $dim_M(Cay(D_{2n},S))\geq 3$.
Now assume that $|S|=3$.
Since $S$ is a generating set and $e\notin S=S^{-1}$, we consider the following cases.
\\
{\bf Case 1.}  $S=\{a^i, a^{-i}, a^jb\}$.
\\ Since $(a^jb)(a^i)^t(a^jb)=a^{-it}$, the order of $a^i$ is ${n\over \gcd(i,n)}$ and $S$ is a generating set, we have $gcd(i,n)=1$.
%For the subgroup $H=\langle a^i\rangle$ let $T=\{g_0=e, g_1=b\}$ be a right transversal for $H$ in $D_n$.
Thus $O(a^i)=n$ and   vertices $a^{ni},a^{(n-1)i},...,a^{2i},a^i$ induce an $n$-cycle in $Cay(D_{2n},S)$.
Since $a^j\in \langle a^i\rangle$, there exists $k\in\{1,2,...,n\}$ such that $a^j=a^{ki}$. Therefore
 $n$ vertices $$a^{ki}b, a^{(k+1)i}b, ..., a^{(k+n-2)i}b, a^{(k+n-1)i}b$$
 induce another cycle in $Cay(D_{2n},S)$.
Now for each $1\leq \ell \leq n$ let $M_{\ell}=\{a^{\ell i}, a^{(k+n-\ell) i}b\}$.  Note that $a^{ni}=e$ and $M_s\cap M_k=\emptyset$ for each $s\neq k$.
Since $a^{\ell i} (a^{(k+n-\ell) i}b)^{-1}=a^{ki}b=a^jb\in S$, two vertices $a^{\ell i}$ and $a^{(k+n-\ell) i}b$ are adjacent in $Cay(D_{2n},S)$.
Thus,  the edges $M_1, M_2, ..., M_n$ provide a perfect matching in $Cay(D_{2n},S)$.
Consequently, $Cay(D_{2n},S)$ is isomorphic to $P_2\times C_n$.
Now Theorem \ref{11} implies that $dim_M(Cay(D_{2n},S))=2$ if and only if $n$ is odd.
\\\\
%In the following two cases, it is shown that $dim(X(D_n,S))\neq 2$. In case 3, we can find a sharp bound for metric dimension of $X(D_n,S)$ and so a new class of graphs with constant metric dimension three and four are defined. In case 4, by Lemma \ref{22}, we prove that metric dimension of $X(D_n,S)$ is greater than two.
{\bf Case 2.}  $S=\{a^{n/2}, a^ib, a^jb\}$ where $n$ is an even number.
%Since $S$ is a generating set for $D_{2n}$, $b\in \langle S\rangle$.
\\ Let $x=a^ib$ and $y=a^jb$. Since  $a^{n/2}$ is in the center of $D_{2n}$ and $O(a^{n\over2})=2$,  $\langle S\rangle=\langle a^ib, a^jb\rangle\cup a^{n/2}\langle a^ib, a^jb\rangle$.
Thus, $a\in \langle a^ib, a^jb\rangle$ or $a\in a^{n/2}\langle a^ib, a^jb\rangle$.
Note that $|\langle a^{n/2}, a^ib\rangle|=|\langle a^{n/2}, a^jb\rangle|=4$. Thus, $a\notin \langle a^{n/2}, a^ib\rangle$ and $a\notin \langle a^{n/2}, a^jb\rangle$.
%The following two subcases will be considered.
\\\\
{\bf Subcase 2.1.} $a\in \langle a^ib, a^jb\rangle$.
\newline
In this case, using Lemma \ref{GenCond} we have $gcd(i-j,n)=1$.  Thus, $O(xy)=O(a^{i-j})=n$ and
 $Cay(D_{2n},S)$ contains a Hamiltonian cycle (on $2n$ vertices) as below.
\begin{equation*}
  e\sim y\sim xy\sim yxy\sim (xy)^2\sim y(xy)^2\sim \ldots \sim y(xy)^{n-1}\sim (xy)^{n}=e.
\end{equation*}
For each divisor $d$ of $n$ the cyclic group $\Bbb{Z}_n$ has unique cyclic subgroup of order $d$.
Since $\langle a^{i-j}\rangle=\langle a\rangle$ and $|\langle a^{(i-j){n\over2}}\rangle|=|\langle a^{n\over2}\rangle|=2$ ,
we have $a^{n/2}=(a^{i-j})^{n/2}$.
For each $1\leq\ell\leq {n\over 2}$ let $M_{\ell}=\left\{(xy)^{\ell},(xy)^{\ell+n/2}\right\}$ and $T_{\ell}=\left\{y(xy)^{\ell},y(xy)^{\ell+n/2}\right\}$.
Note that $M_s\neq M_k$ and $T_s\neq T_k$ for each $s\neq k$. Also, each $M_\ell$ is an edge in $Cay(D_{2n},S)$ because
$$(xy)^{\ell+n/2}(xy)^{-\ell}=(xy)^{n/2}=(a^{i-j})^{n/2}=a^{n/2}\in S$$
Thus, $\{M_1,M_2,...,M_{n\over2}\}$ is a matching in $X(D_n,S)$. Similarly, $\{T_1,T_2,...,T_{n\over2}\}$ is a matching  and hence,
$\{M_1,M_2,...,M_{n\over2}, T_1,T_2,...,T_{n\over2}\}$ provides a perfect matching for $Cay(D_{2n},S)$.
Therefore, we have a cycle on $2n$ vertices in which its opposite pairs of vertices are adjacent (see Figure \ref{fig2} (i)).
This implies that $Cay(D_{2n},S)$ is  a $M\ddot{o}bius~ Ladder$  and  by Theorem \ref{21},  $dim_M(Cay(D_{2n},S))\neq 2$.
\\\\
{\bf Subcase 2.2.} $a\in a^{n/2}\langle a^ib, a^jb\rangle$.
\\
In this case, there exists $k\in \Bbb{Z}$ such that $a=a^{n\over2}(a^{i-j})^k$. Hence, $a^{{n\over2}+1}\in \langle a^{i-j}\rangle$ and  $a^2=(a^{{n\over2}+1})^2\in \langle a^{i-j}\rangle$. Thus, $|\langle a^{i-j}\rangle|\geq |\langle a^2\rangle|={n\over2}$. Hence, $O(a^{i-j})={n\over2}$ or $O(a^{i-j})=n$. The situation $O(a^{i-j})=n$ is considered in Subcase 2.1 and we can assume that
 $O(xy)=O(a^{i-j})=n/2$. Therefore, $Cay(D_{2n},S)$ contains two $n$-cycles as below.
\begin{eqnarray*}
  &e\sim y\sim xy\sim yxy\sim (xy)^2\sim y(xy)^2\sim \ldots \sim y(xy)^{n/2-1}\sim (xy)^{n/2}=e,&\\
&   a^{n/2}\sim ya^{n/2}\sim (xy)a^{n/2}\sim y(xy)a^{n/2}\sim (xy)^2a^{n/2}
        \sim \ldots \sim (xy)^{n/2}a^{n/2}=a^{n/2}.&
\end{eqnarray*}
The fact $O(xy)=n/2$ implies that $n$ vertices apeared in each cycle are distinct.
Also, using the appearence of $b$, it is easy to see that $(xy)^t\neq y(xy)^s a^{n\over 2}$ and $y(xy)^t\neq (xy)^sa^{n\over 2}$ for each $t,s\in\Bbb{Z}$.
If there exist $t,s\in\Bbb{Z}$ such that $(xy)^t=(xy)^sa^{n\over 2}$ or $y(xy)^t=y(xy)^sa^{n\over 2}$, then $a^{n\over 2}=(a^{i-j})^{(t-s)}\in \langle a^{i-j}\rangle =\langle a^2\rangle$.  Thus, $4\mid n$ which is a contradiction (see Lemma \ref{4|n}).
Therefore, all of $2n$ vertices appeared in these cycles are distinct. Since $a^{n\over2}\in S$, corresponding vertices of two cycles are adjacent (see Figure \ref{fig2} (ii)). Hence, $Cay(D_{2n},S)$ is isomorphic to $P_2\times C_n$ and Theorem \ref{11} implies that $dim_M(Cay(D_{2n},S))=3$.
\\ {\bf Case 3.}  $S=\{a^ib, a^jb, a^tb\}$.
\\ Let $H=\langle a\rangle$ and hence,  $V(Cay(D_{2n},S))=H\cup Hb$.
If $a^s, a^t\in H$, then $a^sa^{-t}=a^{s-t}\notin S$. Thus, the subset  $H$ of vertices induces an independent set in $Cay(D_{2n},S)$.
Similarly,  $Hb$ is an independent set.
%Since $|S|=3$,  each vertex in $H$ is adjacent to three vertices in $Hb$ and,  each vertex in $Hb$ is adjacent to three vertices in $H$.
Consequently, $Cay(D_{2n},S)$ is a $3$-regular bipartite graph on $2n$ vertices. Now Lemma \ref{22} implies that  $dim(Cay(D_{2n},S))$ is at least three.
This completes the proof.
\end{proof}
%%%%%%%%%%%%%%%%%%%%%%%%%%%%%%%
%%%%%%%%%%%%%%%%%%%%%%%%%%%%%%%
%%%%%%%%%%%%%%%%%%%%%%%%%%%%%%%
\begin{figure}[h!]
\begin{center}
\[\begin{tikzpicture}
 \vertex (a) at (0:1.5) [label=right:$e$][fill=black] {};
 \vertex (b) at (36:1.5) [label=above right:$y$] [fill=black]{};
 \vertex (c) at (72:1.5) [label=above:$xy$] [fill=black]{};
  \vertex (d) at (108:1.5) [label=above:$yxy$] [fill=black]{};
 \vertex (f) at (144:1.5) [label=above left:$(xy)^2$] [fill=black]{};
\vertex (g) at (180:1.5) [label= left:$(xy)^{n/2}$] [fill=black]{};
 \vertex (aa) at (216:1.5) [label=below left:$y(xy)^{n/2}$][fill=black] {};
 \vertex (bb) at (252:1.5) [label=below left:$(xy)^{n/2+1}$] [fill=black]{};
 \vertex (cc) at (288:1.5) [label=below:$y(xy)^{n/2+1}$] [fill=black]{};
  \vertex (dd) at (324:1.5) [label=below right:$(xy)^{n/2+2}$] [fill=black]{};
   \vertex (ddd) at (-90:2.25) [label=below:$(i)$] [draw=none]{};
  %%%%%%%%%%%%%%%%%%%%%%%%%%%%%%%%%%%%%%%%%%%%%%%%%%%%%%%%%%%%%%%%%%%%%%%%%%%%%%%%%%%%%5
  \vertex (abc) at (7,0)[draw=none]{};
 \vertex (a1) at ([shift={(abc)}]90:1.5) [label=below:$e$][fill=black] {};
 \vertex (b1) at ([shift={(abc)}]45:1.5) [label=below:$y$] [fill=black]{};
 \vertex (c1) at ([shift={(abc)}]360:1.5) [label=left:$xy$] [fill=black]{};
  \vertex (d1) at ([shift={(abc)}]-45:1.5) [label=above left:$yxy$] [fill=black]{};
 \vertex (f1) at ([shift={(abc)}]180:1.5) [label=right:$(xy)^{n/2-1}$] [fill=black]{};
\vertex (g1) at ([shift={(abc)}]135:1.5) [label= below right:$y(xy)^{n/2-1}$] [fill=black]{};
 \vertex (aa1) at ([shift={(abc)}]90:2) [label=above:$\footnotesize{u^{n/2}}$][fill=black] {};
 \vertex (bb1) at ([shift={(abc)}]45:2) [label=above right:$\footnotesize{yu^{n/2}}$] [fill=black]{};
 \vertex (cc1) at ([shift={(abc)}]360:2) [label=right:$\footnotesize{xyu^{n/2}}$] [fill=black]{};
  \vertex (dd1) at ([shift={(abc)}]-45:2) [label=below right:$\footnotesize{yxyu^{n/2}}$] [fill=black]{};
 \vertex (ff1) at ([shift={(abc)}]180:2) [label=left:$\footnotesize{(xy)^{n/2-1}u^{n/2}}$] [fill=black]{};
\vertex (gg1) at ([shift={(abc)}]135:2) [label=above left:$\footnotesize{y(xy)^{n/2-1}u^{n/2}}$] [fill=black]{};
\vertex (gg11) at ([shift={(abc)}]-90:3) [label=above:$(ii)$] [draw=none]{};
 \path
(a) edge (g)
(b) edge (aa)
(c) edge (bb)
(d) edge (cc)
(f) edge (dd)
(a1) edge (aa1)
(b1) edge (bb1)
(c1) edge (cc1)
(d1) edge (dd1)
(f1) edge (ff1)
(g1) edge (gg1);
\draw [dashed] ([shift={(abc)}]-45:2) arc (-45:-180:2);
\draw [ultra thick] ([shift={(abc)}]90:2) arc (90:45:2);
\draw ([shift={(abc)}]45:2) arc (45:0:2);
\draw [ultra thick] ([shift={(abc)}]0:2) arc (0:-45:2);
\draw ([shift={(abc)}]90:2) arc (90:135:2);
\draw [ultra thick] ([shift={(abc)}]135:2) arc (135:180:2);
\draw [ultra thick] ([shift={(abc)}]90:1.5) arc (90:45:1.5);
\draw  ([shift={(abc)}]45:1.5) arc (45:0:1.5);
\draw [ultra thick] ([shift={(abc)}]0:1.5) arc (0:-45:1.5);
\draw  ([shift={(abc)}]90:1.5) arc (90:135:1.5);
\draw [ultra thick] ([shift={(abc)}]135:1.5) arc (135:180:1.5);
\draw
[dashed] ([shift={(abc)}]-45:1.5) arc (-45:-180:1.5);
\draw  ([shift={(abc)}]180:1.5) arc (180:-45:1.5);

\draw [ultra thick] (0:1.5) arc (0:36:1.5);
\draw  (36:1.5) arc (36:72:1.5);
\draw [ultra thick] (72:1.5) arc (72:108:1.5);
\draw  (108:1.5) arc (108:144:1.5);
\draw [dashed] (144:1.5) arc (144:180:1.5);
\draw [ultra thick] (180:1.5) arc (180:216:1.5);
\draw  (216:1.5) arc (216:252:1.5);
\draw [ultra thick] (252:1.5) arc (252:288:1.5);
\draw  (288:1.5) arc (288:324:1.5);
\draw [dashed] (324:1.5) arc (324:360:1.5);
;
\end{tikzpicture}\]
\caption{{\footnotesize $(i)$ A $M\ddot{o}bius~ Ladder$ graph, and $(ii)$ the Cartesian product $P_2\times C_n$.}}
\label{fig2}
\end{center}
\end{figure}
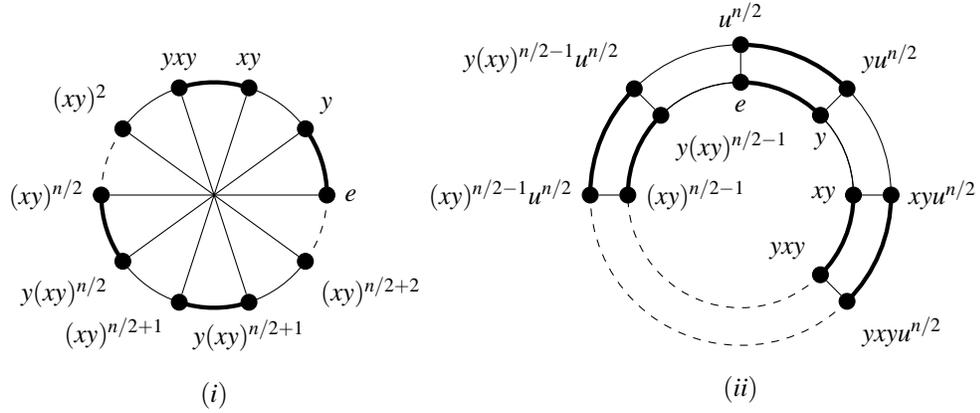
\begin{corollary}
If $S\subseteq D_{2n}$ such that $e\notin S=S^{-1}$ and $|S|\geq 4$, then $\dim_M(Cay(D_{2n},S))\geq 3$.
\end{corollary}
%We have read and understood BMJ policy on declaration of interests and
We declare that we have no competing interests
%\section*{Acknowledgments}
%\section*{References}
%\bibliographystyle{plan}
%\bibliography{mybibfile}

%%%%%%%%%%%%%%%
\end{document}